\documentclass{amsart}
\usepackage{ifthen, citesort}
\usepackage{amssymb,amsmath}
\usepackage[latin1]{inputenc}

\DeclareMathOperator{\diag}{diag}
\DeclareMathOperator{\Ric}{Ric}
\DeclareMathOperator{\Ricci}{Ric}

\def\upi{{ {}^i } }
\def\uph{{ {}^h } }
\def\up#1{{ {}^{#1} } }
\def\boundary{\partial}
\def\R{\mathbb{R}}
\def\N{\mathbb{N}}
\def§{\mathbb{S}}
\def\of{\circ}

\def\timess{×}
\def\gradh{    {{}^{{}^h} \! \nabla}    }

\def\fracp#1#2{\frac{\partial #1}{\partial #2}}
\def\theta{\vartheta}
\def\partt{\frac{d}{d t}}
\def\phi{\varphi}

\def\parti#1#2{\frac{\partial #1 } {\partial #2} }
\def\de{\delta}
\def\be{\beta}
\def\al{\alpha}
\def\ga{\gamma}

\def\epsilon{\varepsilon}

\long\def\umbruch{{\displaybreak[1]}}

\def\ol#1{\overline{#1}}
\def\gradsquared{\sum\limits_{i,\,j,\,k}
\left(\gradh_k g_{ij}\right)^2}
\def\dt{\frac{d}{dt}}
\def\delt{\frac{\partial}{\partial t}}
\def\ep{\epsilon}
\def\grad{\nabla}

\def\ti{\tilde}

\def\M{\mathcal M}
\def\Mloc{\mathcal M_{\text{\it{loc}}}}

\mathchardef\ordinarycolon\mathcode`\:
  \mathcode`\:=\string"8000
  \begingroup \catcode`\:=\active
    \gdef:{\mathrel{\mathop\ordinarycolon}}
  \endgroup

\newtheorem{theorem}{Theorem}[section]
\newtheorem{lemma}[theorem]{Lemma}
\newtheorem{corollary}[theorem]{Corollary}

\theoremstyle{definition}
\newtheorem{remark}[theorem]{{Remark}}
\newtheorem{definition}[theorem]{{Definition}}

\numberwithin{equation}{section}
\begin{document}

\title{Stability of Euclidean space under Ricci flow}

%\thanks{}

\author{Oliver C. Schnürer}
\address{Oliver Schnürer:
  Freie Universität Berlin, Arnimallee 6, 
  14195 Berlin, Germany}
\curraddr{}
\email{Oliver.Schnuerer\fuhome}


\author{Felix Schulze}
\address{Felix Schulze: 
  Freie Universität Berlin, Arnimallee 6, 
  14195 Berlin, Germany}
\curraddr{}
\email{Felix.Schulze\fuhome}
\def\fuhome{@math.fu-berlin.de}

\author{Miles Simon}
\address{Miles Simon: Universität Freiburg, Eckerstraß{}e 1,
  79104 Freiburg i. Br., Germany}
\curraddr{}
\email{msimon@gmx.de}

%    General info
\subjclass[2000]{53C44, 35B35}
% 53C44 Geometric evolution equations (mean curvature flow)
% 35Bxx Qualitative properties of solutions
% 35B35 Stability, boundedness
\date{June 2007.}

\dedicatory{}

\keywords{Stability, Ricci flow.}

\begin{abstract}
  We study the Ricci flow for initial metrics which are $C^0$ small
  perturbations of the Euclidean metric on $\R^n$. In the case that
  this metric is asymptotically Euclidean, we show that a Ricci
  harmonic map heat flow exists for all times, and converges uniformly
  to the Euclidean metric as time approaches infinity. In proving this
  stability result, we introduce a monotone integral quantity which
  measures the deviation of the evolving metric from the Euclidean
  metric. We also investigate the convergence of the diffeomorphisms
  relating Ricci harmonic map heat flow to Ricci flow.
\end{abstract}

\maketitle

%%%%%%%%%%%%%%%%%%%%%%%%%%%%%%%%%%%%%%%%%%%%%%%%%%%%%%%%%%%%
\section{Introduction}
In this paper we investigate the evolution of a family of complete
non-compact Riemannian manifolds $\left(\R^n,g(t)\right)$ under 
Ricci flow
\begin{equation}\label{ricci flow}\tag{RF}
\begin{cases}
\,\delt g(t)=-2\Ric(t)&\text{in }\R^n×(0,\infty),\\
\,g(t)\to g_0&\text{in }C^0_{\text{\it{loc}}}\text{ as }t\searrow0,
\end{cases}
\end{equation}
where $g_0$ is a given initial metric on $\R^n$. We 
study the long-term behavior as $t\to\infty$ of solutions
to \eqref{ricci flow} for initial metrics $g_0$ which are
$C^0$-close to the standard Euclidean metric $h$, 
in standard coordinates $h_{ij}=\delta_{ij}$. For analytic 
reasons, it is convenient to study the Ricci harmonic map heat
flow which is a variant of the Ricci-DeTurck flow 
\begin{equation}\label{DeTurck flow}
\begin{cases}
\,\delt g_{ij}(x,t)=-2R_{ij}(x,t)+\nabla_iV_j+\nabla_jV_i&
\text{in }\R^n\timess[0,\infty),\\
\,g_{ij}(t)\to(g_0)_{ij}&\text{in }C^0_{\it{loc}}\left(\R^n\right)
\text{ as }t\searrow0,
\end{cases}\end{equation}
where
$V_i=g_{ik}\left({}^g\Gamma^k_{rs}-{}^h\Gamma^k_{rs}\right)g^{rs}$,
for a flat metric $h$, $V_i=g_{ik}{}^g\Gamma^k_{rs}g^{rs}$.  If a
family of metrics $\left(g(t)\right)_{t\in(0,\,\infty)}$ solves
\eqref{DeTurck flow}, we call it a solution to the $h$-flow with
initial metric $g_0$.  Precise definitions are to be found in Section
\ref{nota sec}.  Note that the $h$-flow \eqref{DeTurck flow} and the
Ricci flow \eqref{ricci flow} are equivalent up to diffeomorphisms
\cite{ShiJDG1989} in the case that the initial metric is smooth, and
$M$ is compact.
\begin{definition}
Let $g_1$ and $g_2$ denote two Riemannian metrics on a given
manifold.  We say that $g_1$ is $\delta$-fair to $g_2$, if
$$\delta^{-1}g_2\leq g_1\leq\delta\,g_2\,.$$
We call $g_1$ $\epsilon$-close to $g_2$, if $g_1$ is
$(1+\epsilon)$-fair to $g_2$.  
\end{definition}

We denote with $\mathcal M^k(\R^n,I)$ the space of families
$(g(t))_{t\in I}$ of sections in the space of Riemannian metrics on
$\R^n$ which are $C^k$ on $\R^n\times I$. Similarly, we define
$\M^\infty$, $\Mloc^k$ and use $\mathcal M^k(\R^n)$ if the metric does
not depend on $t$. We wish to point out that we use $C^k$ on
non-compact sets to denote the space, where derivatives of order up to
$k$ are in $L^\infty$. We also use $C^k_{\text{\it{loc}}}$.

Our existence result reads as follows
\begin{theorem}\label{thm one}
For every $\epsilon>0$ there exists
$0<\epsilon_0(\epsilon,n)\leq \epsilon,$ such
that the following is true: If $g_0\in\M^0(\R^n)$ is a Riemannian
metric on $\R^n$ which is $\epsilon_0$-close to the standard 
metric $h=\delta$, then there exists
a solution $g\in\M^\infty\big(\R^n,
(0,\infty)\big)\cap\Mloc^0\big(\R^n,[0,\infty)\big)$ to 
\eqref{DeTurck flow} such that $g(t)$ is $\epsilon$-close
to $h$ for all $t$. As $t\to\infty$, the metrics $g(t)$ converge 
subsequentially in $\Mloc^\infty\big(\R^n\big)$
to a complete flat metric.
\end{theorem}

In this situation, our main theorem addresses convergence
to the Euclidean background metric.
\begin{theorem}\label{thm conv}
  Let $g\in\M^\infty\big(\R^n,
  (0,\infty)\big)\cap\Mloc^0\big(\R^n,[0,\infty)\big)$
  be a solution to \eqref{DeTurck flow}.  Assume that $g(t)$ is
  $\tilde{\epsilon}(n)$-close to the standard metric $h$ for all
  $t\geq 0$, for some $\tilde{\epsilon}(n)$ chosen sufficiently small,
  and that $g_0$ is $\epsilon(r)$-close to $h$ on $\R^n\setminus
  B_r(0)$ with $\epsilon(r)\to0$ as $r\to\infty$. Then
$$g(t)\to h\quad\text{in }C^\infty\text{ as }t\to\infty.$$
That is:
 $$\sup_{\R^n} | g(t)- h | \to 0$$ as $t \to \infty$, and 
 $$\sup_{\R^n} \left| {\uph \grad}^k g(t) \right| \to 0$$ as $t \to
 \infty$ for all 
 $k \in \{1,2 \ldots \}.$
\end{theorem}

In this paper we will be concerned with the geometric quantities
\begin{equation}\label{phi def eqn}
\phi_m:=g^{a_1b_1}h_{b_1a_2}g^{a_2b_2}h_{b_2a_3}\cdots
g^{a_mb_m}h_{b_ma_1}
\end{equation}
and 
\begin{equation}\label{psi def eqn}
\psi_m:=g_{a_1b_1}h^{b_1a_2}g_{a_2b_2}h^{b_2a_3}\cdots
g_{a_mb_m}h^{b_ma_1}.
\end{equation}
If $(\lambda_i)$ are the eigenvalues of $(g_{ij})$ with respect
to $(h_{ij})$, we get
$$\phi_m=\sum\limits_{i=1}^n\frac1{\lambda_i^m}\quad\text{and}
\quad\psi_m=\sum\limits_{i=1}^n\lambda_i^m.$$
In particular
\begin{equation}\label{phi psi 2n sum}
\phi_m+\psi_m-2n=\sum\limits_{i=1}^n\left(\frac1{\lambda_i^m}
+\lambda_i^m-2\right)=\sum\limits_{i=1}^n
\frac1{\lambda_i^m}\left(\lambda_i^m-1\right)^2\geq0
\end{equation}
is non-negative and vanishes precisely when $\lambda_i=1$ for
all $i\in\{1,\ldots,n\}$. 

\begin{theorem}\label{thm two}
  Let $n\geq3$ and $g\in
  \Mloc^\infty\big(\R^n,[0,\infty)\big)$ be a solution
  to \eqref{DeTurck flow}, satisfying all the conditions of Theorem
  \ref{thm conv}. Then there exist $m=m(n)\in\N$ and
  $\tilde\epsilon(n)>0$, such that the following holds.  If
  $1\leq p<\frac n2$ can be chosen such that initially
\begin{equation}\label{int p cond}
\int\limits_{\R^n}(\phi_m+\psi_m-2n)^p<\infty,
\end{equation}
then there exists a smooth family $(\phi_t)_{t\geq 0}$ of
diffeomorphisms of $\R^n$, $\phi_0=\rm{id}_{\R^n}$, such that for
$\tilde{g}(t):=\phi_t^*g(t)$ the family $(\tilde{g}(t))_{t\geq 0}$ is
a smooth solution to \eqref{ricci flow} satisfying
$$\tilde{g}(t)\to(\phi_\infty)^*h\quad\text{in } 
\M^\infty\big(\R^n\big)\text{ as }t\to\infty$$ for some smooth
diffeomorphism $\phi_\infty$ of $\R^n$ which satisfies
$\phi_t\to\phi_\infty$ in $C^\infty\left(\R^n,\R^n\right)$ as
$t\to\infty$ and
$$|\phi_\infty(x)-x|\to0\quad\text{as }|x|\to\infty.$$
\end{theorem}

Rewriting \eqref{int p cond} in terms of a decay condition
on $\epsilon(r)$ in the closeness condition for $g_0$,
see Theorem \ref{thm conv}, yields
\begin{corollary}
  Let $n\geq3$ and $g\in \Mloc^\infty\big(\R^n,[0,\infty)\big)$ be a
  solution to \eqref{DeTurck flow}, satisfying all the conditions of
  Theorem \ref{thm conv}. If $g_0$ is $\epsilon_0$-close to $h$ for
  some $\epsilon_0=\epsilon_0(n)>0$ sufficiently small, and in
  addition there exist constants $C>0$ and $\zeta>0$ such that $g_0$
  is $Cr^{-1-\zeta}$-close to $h$ on $\R^n\setminus B_r(0)$, then the
  same conclusions as in Theorem \ref{thm two} hold.
\end{corollary}

Note that the solutions constructed in Theorem \ref{thm one} satisfy
the conditions of Theorem \ref{thm conv}, but are not necessarily
$\epsilon(r)$-close to $h$ at infinity.
We point out that these solutions
may depend on the subsequence chosen in the construction 
in Section \ref{existence sec}. Note, however, that it is not
clear, whether solutions to the initial value problem
\eqref{DeTurck flow}
are unique, as in general $g_0$ does not have bounded
curvature or a well-defined Riemannian curvature tensor.
Even if $g(t)$ is $\epsilon(n)$-close to $h$ on $\R^n$
for all $t>0$, we do not know, whether such solutions are
unique. If we study only solutions as constructed in 
Section \ref{existence sec}, we can replace the 
$\epsilon(r)$-closeness condition imposed on $g_0$ by a
considerably weaker integrability condition.

\begin{theorem}\label{thm three}
  Fix $p\geq1$, $m=m(n)\in\N$ sufficiently large, and
  $\epsilon_0=\epsilon_0(n)>0$ sufficiently small.  Let
  $g_0\in\M^0\big(\R^n\big)$ be a Riemannian metric which is
  $\epsilon_0$-close to the standard metric $h$. Let $g \in
  \Mloc^\infty\left(\R^n,(0,\,\infty)\right)$ be a solution to
  \eqref{DeTurck flow} as constructed in Section \ref{existence sec}.
  If for every $\delta>0$ the integral $I^{m,p}_\delta(0)$, as defined
  in Theorem \ref{int delta thm}, is finite, then
$$g(t)\to h\quad\text{in }\M^\infty\left(\R^n\right)
\text{ as }t\to\infty.$$
\end{theorem}

In the situation of Theorem \ref{thm three}, 
if $I_0^{m,p}(0)$ is finite, we can show that the
diffeomorphisms $\phi_t$ stay bounded and converge in
$C^\infty_\text{{\it loc}}$ to a
limiting diffeomorphism $\phi_\infty$.

\textbf{Outline.}
In order to prove stability of Euclidean $\R^n$ under the
$h$-flow, we proceed as follows. 

If we perturb the metric on $\R^n$, such that the
perturbed metric is $C^0$-close to our original metric and
the perturbation is small enough, this is preserved for all
times. The perturbation is measured in terms of the deviation 
of the eigenvalues $(\lambda_i)$ of the perturbed metric with 
respect to the Euclidean metric from one. We have
$$(1+\epsilon)^{-1}\leq\lambda_i\leq1+\epsilon
\quad\text{for}\quad i=1,\,\ldots,\,n$$
everywhere for all times and $0<\epsilon=\epsilon(n)$ 
provided that initially
$(1+\epsilon_0)^{-1}\leq\lambda_i\leq1+\epsilon_0$ for some
$0<\epsilon_0\leq\epsilon_0(n)$. 

We have interior estimates for the gradient of the metric 
evolving under DeTurck flow. If we assume that 
$\lambda_1\leq\ldots\leq\lambda_n$, we get the estimate
$$|\lambda_i(x,t)-\lambda_i(y,t)|\leq\frac c{\sqrt t} 
\cdot d_\text{Eucl.}(x,y).$$
Thus, for large times, the eigenvalues $\lambda_i$ are
almost spatially constant in any Euclidean ball. 
It is, however, possible that these spatial constants
vary in time. In particular, it is not clear up until now
whether these constants converge as time tends to infinity.

Our perturbation of the Euclidean
metric is small near infinity, i.\,e.{} all eigenvalues
$\lambda_i$ approach one near infinity. Note, however,
that we do not impose a decay rate on our initial metric, 
at which the eigenvalues converge to one as we approach 
spatial infinity. Still, we obtain for all positive times that 
the eigenvalues converge to one at spatial infinity.

In order to get uniform control on the eigenvalues,
we use the quantity considered in \eqref{phi psi 2n sum}
for some $m\in\N$. It vanishes precisely when 
$\lambda_i=1$ for all $i$ and measures the deviation from 
$\lambda_i=1$. The functions $\phi_m$ and $\psi_m$ 
were initially considered by W.-X. Shi in
\cite{ShiJDG1989}.              
The key estimate is to show essentially that
$\int(\phi_m+\psi_m-2n)$ is non-increasing in time. 
Here we integrate over the manifold at a fixed time. For technical
reasons we have to consider a more complicated quantity. For details
we refer to the respective proofs in Section \ref{integral est sec}.

This implies convergence of the eigenvalues $\lambda_i(x,t)$ to $1$ as
$t\to\infty$, uniformly in $x$. For if this were not the case, we
could pick points $(x_k,t_k)$ with $t_k\to\infty$ such that at least
one $\lambda_i(x_k,t_k)$ differs significantly from $1$. If $t_k$ is
large enough, we find a big ball, where at least one eigenvalue
differs considerably from $1$. This yields an arbitrarily large
contribution to $\int(\phi_m+\psi_m-2n)$ if $t_k$ is big
compared to $\max_i|\lambda_i(x_k,t_k)-1|$, contradicting the
monotonicity of $\int(\phi_m+\psi_m-2n)$.  Thus
$\lambda_i(x,t)\to1$, uniformly in $x$ as $t\to\infty$, follows.

\begin{remark}
Note that all the above results also hold if we replace
$\R^n$ by $\mathbb T^k×\R^{n-k}$ for $n-k\geq1$, where
$\mathbb T^k$ is a flat $k$-dimensional torus.   
\end{remark}

In two space dimensions, Ricci flow is a conformal flow given by the
evolution equation
$$\delt g=-Rg,$$
where $R$ is the scalar curvature, see \cite{HamiltonTwo}. In this
situation, we obtain (for details see Appendix \ref{2d sec})
\begin{theorem}
  Let $g_0=e^{-u_0}h$, where $h=\delta$ is the standard Euclidean
  metric, and $u_0\in C^0\left(\R^2\right)$ such that 
$$\sup\limits_{\R^2\setminus B_r(0)}|u|\to0\quad\text{as
}r\to\infty.$$ Then there exists a smooth solution
$(g(t))_{t\in(0,\infty)}$, $g(t)=e^{-u(\cdot,t)}h$, to Ricci flow such
that $u(\cdot,t)\to u_0$ in $C^0_{\text{\it{loc}}}$ as $t \searrow 0$.
Furthermore, as $t\to\infty$, we obtain that $u(\cdot,t)\to0$ in
$C^\infty\left(\R^2\right)$.
\end{theorem}

For two dimensional Ricci flow, L.-F. Wu studied long time behavior
and convergence of solutions \cite{LaniWuRtwo}. She studied initial
complete metrics $g_0 = e^{u_0} h$ which have bounded curvature and
satisfy $e^{-u_0}|Du_0|^2 < \infty.$ In this case, she showed that a
long time solution to Ricci flow exists and that it ``converges
smoothly in the sense of modified subsequences'' to a smooth limiting
metric as $t \to \infty.$ A smooth family of metrics $(g(t))_{t \in
  [0,\infty)}$ on $\R^n$ ``converges smoothly in the sense of modified
subsequences'' to a metric $l$ on $\R^n$ as $t\to\infty$ if there
exist diffeomorphisms $\phi_i: \R^n \to \R^n$ and a sequence $t_1 <
t_2 < \ldots$ with $t_i \to \infty$ as $i \to \infty$, such that
$(\phi_i)^* (g(t_i)) \to l$ smoothly on any fixed compact subset of
$\R^n$. In particular, this does not imply uniform convergence on all
of $\R^n$.\par By imposing strong decay conditions on the curvature
tensor at infinity, stability of flat space under Ricci flow was
proved by W.-X.  Shi \cite{ShiJDG1989Konv}.  Stability of compact flat
manifolds under Ricci flow was studied by C. Guenther, J. Isenberg, D.
Knopf and by N.  \v{S}e\v{s}um
\cite{SesumFlatRicciStab,StabCompactRicciFlat}.  In the rotationally
symmetric situation stability of flat Euclidean space was investigated
by T. Oliynyk and E. Woolgar \cite{OliynykWoolgar}.  A.  Chau and the
first author obtained a stability result for the Kähler potential of
stationary rotationally symmetric solitons of positive holomorphic
bisectional curvature under Kähler-Ricci flow \cite{OSAlbert}.
Uniqueness of solutions to Ricci flow with bounded curvature on
non-compact manifolds is discussed in
\cite{ChenZhuRicciUnique,HsuRicciUnique}.  J. Clutterbuck and the
first two authors proved stability of convex rotationally symmetric
translating solutions to mean curvature flow in
\cite{JCOSFSMCFStability}.  Short time existence results for
$C^0$-metrics were shown in \cite{MilesC0} using similar techniques to
this paper.

The rest of the paper is organized as follows. We introduce our
notation in Section \ref{nota sec} and recall some evolution
equations in Section \ref{evol eq sec}. Existence of solutions
for all times is shown, see Section \ref{existence sec}. We prove
interior closeness and a priori estimates in Section \ref{int est sec}.
In Section \ref{integral est sec} we show that an integral quantity 
based on the expression in \eqref{phi psi 2n sum} is monotonically 
decreasing along the flow. Combining this with the interior gradient
estimates, we obtain in Section \ref{eigenval conv sec} that the 
eigenvalues of $g_{ij}(x,t)$ with respect to the background metric
converge uniformly to one as $t\to\infty$. An iteration scheme allows us
to improve our gradient estimates for large times, see Section
\ref{imp est sec}. In Section \ref{const sol sec}, 
we study the diffeomorphisms relating \eqref{ricci flow} and 
\eqref{DeTurck flow}, and show their convergence for large times. 
We address the proof of Theorem \ref{thm three} 
in Section \ref{conv int bounds}. Stability results for Ricci flow in
two dimensions are addressed in Appendix \ref{2d sec}. 

We want to thank Gerhard Huisken for fruitful discussions. 
The third author acknowledges support from SFB 647/B3 during 
his visits in Berlin.  
This paper is part of the research project ``Stability of non-compact
manifolds under curvature flows'' of the first two authors within the
priority program ``Global Differential Geometry'' SPP 1154 of the
German research foundation DFG.

\section{Notation}\label{nota sec}
Let $g=(g_{ij})_{1\leq i,\,j\leq n}$ be a Riemannian metric.
By $(g(t))_{t\in(0,\,\infty)}$, we denote a family of metrics. 
We denote the inverse of $(g_{ij})$ by $\left(g^{ij}\right)$
and use the Einstein summation convention for repeated 
upper and lower indices. If $g_1$ and $g_2$ are two metrics
such that $(g_1)_{ij}\xi^i\xi^j\leq(g_2)_{ij}\xi^i\xi^j$
for all $\left(\xi^i\right)\in\R^n$,
we denote this by $g_1\leq g_2$. Unless otherwise stated,
geometric quantities like covariant derivatives $\nabla_i$ 
and Christoffel symbols $\Gamma^k_{ij}$ are computed with
respect to the evolving metrics $g(t)$. We use indices 
$h$ to denote quantities depending on the metric
$h$, e.\,g.{} $\gradh g$ denotes the 
covariant derivative of $g$ with respect to $h$,
i.\,e.\ a partial derivative. In short formulae, we
also use $\nabla g$ instead. 
With the exception of the beginning of Section \ref{existence sec},
$h$ will always denote the standard metric on $\R^n$. On $\R^n$ we
will always use coordinates such that $h_{ij}=\delta_{ij}$. 
A ball of radius $r$, centered at $x$, is denoted by
$B_r(x)$. We will only use Euclidean balls. 
The norms $|\cdot|$ and $\up h|\cdot|$ are computed with 
respect to the flat background metric $h$.
The letter $c$ denotes generic constants.

In formulae, where we use $\lambda_i$ to denote the eigenvalues
of $g_{ij}$ with respect to $h_{ij}$, we will always assume 
that $h_{ij}=\delta_{ij}$ in the coordinate system chosen.
Several times, we will use that $\uph|g_{ij}(x,t)-g_{ij}(y,t)|\leq A$
implies that $|\lambda_i(x,t)-\lambda_i(y,t)|\leq A$ if we
assume that $\lambda_1(\cdot,t)\leq\ldots\leq\lambda_n(\cdot,t)$.
Similarly, $\uph|g_{ij}(x,t)-h_{ij}|\leq A$
implies $|\lambda_i(x,t)-1|\leq A$.

\section{Evolution Equations}\label{evol eq sec}
In this section, we collect some evolution equations from 
\cite{ShiJDG1989} and state some direct consequences. 

Assume that in appropriate coordinates, we have at a fixed point
and at a fixed time $h_{ij}=\delta_{ij}$, 
$g_{ij}=\diag(\lambda_1,\,\lambda_2,\,\ldots,\,\lambda_n)$,
$\lambda_i>0$.

The evolution equation for the metric, computed with respect to
a flat background metric, is, see 
\cite[Lemma 2.1]{ShiJDG1989},
\begin{align*}
\fracp{}tg_{ij}=&\,g^{ab}\gradh_a\gradh_bg_{ij}\\
&\,+\tfrac12g^{ab}g^{pq}\left(\gradh_i g_{pa}\gradh_jg_{qb}
+2\gradh_ag_{jp}\gradh_qg_{ib}-2\gradh_ag_{jp}
\gradh_bg_{iq}\right.\\
&\,\left.\qquad\qquad\qquad-2\gradh_jg_{pa}\gradh_bg_{iq}
-2\gradh_ig_{pa}\gradh_bg_{jq}\right).
\end{align*}

Consider $\phi_m$ as defined in \eqref{phi def eqn}.
The evolution equation of $\phi_m$ 
for a flat background metric $h_{ij}$ is, 
see \cite[Lemma 2.2 (70)]{ShiJDG1989},
\begin{align*}
\fracp{}t\phi_m=&\,g^{ab}\gradh_a\gradh_b\phi_m\umbruch\\
&\,-m\sum\limits_{i,j,a}\frac1{\lambda_a}
\left[\sum\limits_{k=2}^m
\left(\frac1{\lambda_i}\right)^k\left(\frac1{\lambda_j}\right)^{m+2-k}
\right]\left(\gradh_ag_{ij}\right)^2\umbruch\\
&\,-\sum\limits_{i,q,k}\frac m{2\lambda_i^{m+1}\lambda_q\lambda_k}
\left(\gradh_kg_{iq}+\gradh_qg_{ik}-\gradh_ig_{qk}
\right)^2.
\end{align*}
We deduce for $g(t)$ that is $\epsilon$-close to $h$ that
\begin{align*}
\fracp{}t\phi_m\leq&\,g^{ij}\gradh_i\gradh_j\phi_m\\
&\,-m(m-1)(1+\epsilon)^{-(m+3)}
\sum\limits_{i,\,j,\,k}\left(\gradh_kg_{ij}\right)^2.
\end{align*}

{}From the evolution equation of the metric, we get for 
$\psi_m$, as introduced in \eqref{psi def eqn},
the evolution equation
\begin{align*}
\fracp{}t{\psi_m}=&\,mh^{a_1b_1}g_{b_1a_2}h^{a_2b_2}g_{b_2a_3}\cdots
h^{a_mb_m}\fracp{}tg_{b_ma_1}\umbruch\\
=&\,mh^{a_1b_1}g_{b_1a_2}h^{a_2b_2}g_{b_2a_3}\cdots
h^{a_mb_m}\left(g^{ij}\gradh_i\gradh_jg_{b_ma_1}\right)\\
&\,+c(n)m\underbrace{h^{-1}*\cdots*h^{-1}}_m*
\underbrace{g*\cdots*g}_{m-1}*g^{-1}*g^{-1}
*\gradh g*\gradh g\umbruch\\
=&\,g^{ij}\gradh_i\gradh_j\psi_m\\
&\,-m\sum\limits_{i,\,j,\,a}
\left[\sum\limits_{k=0}^{m-2}\lambda_i^k\lambda_j^{m-2-k}\right]
\frac1{\lambda_a}\left(\gradh_ag_{ij}\right)^2\\
&\,+c(n)m\underbrace{h^{-1}*\cdots*h^{-1}}_m*
\underbrace{g*\cdots*g}_{m-1}*g^{-1}*g^{-1}
*\gradh g*\gradh g,
\end{align*}
where $*$ indicates contractions and linear combinations of
contractions. The factor $c(n)$ indicates that the number of the
respective terms depends only on $n$. Therefore, we deduce
that
\begin{align*}
\fracp{}t\psi_m\leq&\,g^{ij}\gradh_i\gradh_j\psi_m\umbruch\\
&\,-m(m-1)(1+\epsilon)^{-(m-1)}
\sum\limits_{i,\,j,\,k}\left(\gradh_kg_{ij}\right)^2\\
&\,+c(n)m(1+\epsilon)^{m+1}
\sum\limits_{i,\,j,\,k}\left(\gradh_kg_{ij}\right)^2.
\end{align*}

\section{Existence}\label{existence sec}
In this section we prove long time existence of the $h$-flow for a 
smooth initial metric $g_0$ which is $\ep$-close to $h=\delta$.
Short time existence for Ricci-DeTurck flow on non-compact manifolds
 was first proved in 
\cite{ShiJDG1989}.
Short time existence for an arbitrary smooth background metric $h$
with bounded curvature, 
when $g_0$ is $\ep$-close to $h$, was proved in \cite{MilesC0}, 
using similar techniques to those of W.-X. Shi.
We also reprove the short time existence here for completeness. 
(note: some of the citations in \cite{MilesC0} in the proof of 
the short time existence were incorrect.)

\begin{lemma}\label{gradient bounds for Dirichlet solutions}
Let $(g(t))_{t \in [0, T)}$ be a smooth $\ep$-close solution to
the $h$-flow with flat background metric $h=\delta$
on a closed ball $D \subset \R^n$ with
$0<\epsilon\leq\epsilon(n)$ and 
$g(\cdot,t)|_{\boundary D} = h(\cdot)|_{\boundary D}$.
Then
$$\sup\limits_{D×[0,T)}  \Big|\gradh^m g\Big|^2 \leq c(n,m,D,T,g(0)).$$ 
\end{lemma}
\begin{proof}
  We proceed exactly as in the proof of Shi \cite[Lemma
  3.1]{ShiJDG1989}: where he uses $\ti g$, we use $h$.  The only other
  minor difference is that the dependence on $g(0)$ does not
  explicitly appear in Shi's paper \cite[Lemma 3.1]{ShiJDG1989}.  This
  is because he has $\ti g = g(0).$
\end{proof}

\begin{lemma}\label{fairness pres lem}
  Let $(g(t))_{t\in[0,T)}$ be a smooth solution to the $h$-flow on
  $B_i$, $i\in\N$, where $h$ is the standard metric on $\R^n$.  Then
  for all $\epsilon>0$ there exists
  $0<\epsilon_0(n,\epsilon)<\epsilon$, such that the following holds.
  If $g(0)$ is $\epsilon_0$-close to $h$ and $g(t)=h$ on $\partial
  B_i×[0,T)$, then $g(t)$ is $\ep$-close to $h$ for all $t\in[0,T)$.
\end{lemma}
\begin{proof}
Consider
$\Phi:=\phi_m+\psi_m-2n$ as in \eqref{phi psi 2n sum}.
It vanishes on $\partial D×[0,T)$. 
The evolution equations in Section \ref{evol eq sec} imply that
\begin{align}\label{Phi ungl}
\delt\Phi\leq&\,g^{ij}\gradh_i\gradh_j\Phi
\intertext{as long as $g(t)$ is $\tilde\epsilon$-close to $h$ 
for $\tilde\epsilon>0$ satisfying}
c(n)(1+\tilde\epsilon)^{2m}\leq&\,(m-1).\nonumber
\end{align}
Choose $\tilde\epsilon=\tilde\epsilon(n)$ and $m=m(n)$ accordingly.
\par
Fix $\delta>0$ so that $\Phi\leq2\delta$ implies, 
see \eqref{phi psi 2n sum}, 
that the metric is $\epsilon$-close to $h$,
in particular $\tilde\epsilon$-close to $h$.
Now fix $\epsilon_0=\epsilon_0(n)>0$ such that $\epsilon_0<\epsilon$ 
and so that $\Phi\leq\delta$ for every metric which is 
$\epsilon_0$-close to $h$. 
\par
Consider the maximal
time interval $I\subset[0,T)$ on which $\Phi\leq2\delta$.
We may assume that $I=[0,\tau]$ with $\tau<T$.
This implies that $(g(t))_{t\in[0,\tau]}$ is $\tilde\epsilon$-close 
to $h$ (even $\epsilon$-close to $h$).
According to \eqref{Phi ungl}
and the maximum principle,
$\max\limits_{x\in D}\Phi(x,t)$ is non-increasing in $t$
for all $t\in[0,\tau]$.
Thus $\Phi(\tau)\leq\Phi(0)\leq\delta$. 
This contradicts the choice of $\tau$. Therefore 
$I=[0,T)$ and $(g(t))_{t\in[0,T)}$ is $\epsilon$-close to $h$.
\end{proof}

 In the remainder of this section, we will denote with
$\epsilon_0(n,\epsilon)>0$ the constant in Lemma \ref{fairness pres
  lem}. Note that although we allow for arbitrary $\epsilon>0$, it
follows from the proof that $\epsilon_0(n,\epsilon)\le\ol\epsilon$
for some small $\ol\epsilon=\ol\epsilon(n)>0$. 

\begin{theorem}\label{existence thm (Dirichlet)}
  Let $\ep >0$ be given, and $g_0$ be smooth and $\ep_0(\ep,n)$-close
  to $h= \de$ on a closed ball $D \subset \R^n$. Assume that $g_0=h$
  near $\partial D$. Then there exists a unique smooth solution
  $(g(t))_{t \in [0, \infty)}$ to the $h$-flow with $g(0) = g_0$ and
  $g(t) |_{\boundary D} = h|_{\boundary D}$. Furthermore, $(g(t))_{t
    \in [0,\infty)}$ is $\ep$-close to $h$.
\end{theorem}

\begin{proof}
  Lemma \ref{fairness pres lem} implies that a prospective solution
  $g(t)$ is $\epsilon$-close to $h$ as long as it exists.
  Therefore, we can apply the a priori bounds of Lemma \ref{gradient
    bounds for Dirichlet solutions} on a bounded time interval.
  We may then use the same arguments as in \cite[Theorem 7.1, Chapter
  VII]{LadySoloUral} to show that a smooth solution exists (the
  argument used there is based on the Leray-Schauder fixed point
  argument of \cite[Theorem 6.1 of Chapter V]{LadySoloUral}). Note
  that this argument works for every finite time interval.
\end{proof}

\begin{theorem}\label{interior estimates}
  Let $g\in\M^\infty\left(B_R,[0,T]\right)$ be $\tilde\ep$-close to
  $h$, solving \eqref{DeTurck flow}, where
  $\tilde\epsilon=\tilde\epsilon(n)>0$ is sufficiently small and $h =
  \de$. Then
$$ \sup_{B_{\frac R 2} \timess [0,T] }\left|\gradh^m g \right|^2 
\leq c(m,n, g(0)|_{B_R},T).$$
\end{theorem}

\begin{proof}
This follows using  the arguments of Shi 
\cite[Lemma 4.1 and Lemma 4.2]{ShiJDG1989}. 
\end{proof}

\begin{theorem}\label{existence thm}
  Let $\epsilon>0$ be given, and $g_0$ be smooth and
  $\ep_0(n,\ep)$-close to the standard Riemannian metric $h=\de$ on
  $\R^n$.  Then there exists a $\ep$-close solution
  $g\in\Mloc^\infty\left(\R^n, [0, \infty)\right)$ to the $h$-flow
  \eqref{DeTurck flow}.
\end{theorem}

\begin{proof}
  Let $B_i$ be the balls of radius $i$ and center $0$ (with respect to
  the Euclidean metric).  Set $\upi g_0 = \eta_i g_0 + (1 - \eta_i)h,$
  where $\eta_i:\R^n \to \R$ is smooth, and satisfies $\eta_i = 0 $ on
  $B_i - B_{i-1}$, $\eta_i = 1$ on $B_{i-2},$ $0 \leq \eta \leq 1$ and
  $|\grad^m \eta_i|^2 \leq c(m,n).$ Note that $\upi g_0$ is
  $\ep_0$-close to $h$. Let $\upi g(t)\in\M^\infty\left(B_i,[0,
    \infty)\right)$ be the solution coming from the local existence
  theorem (Theorem \ref{existence thm (Dirichlet)}) above.  Using the
  interior estimates of Theorem \ref{interior estimates} above, we see
  that the solutions all satisfy
$$\sup_{B_j \timess [0,T]}  \left|\grad^m \left(\partt\right)^k 
  \left(\upi g\right)\right|^2 \leq c(j,g_0|_{B_{2j}},m,k,n,T),$$ for
all $i $ big enough, and so, using a diagonal subsequence argument, we
can find a subsequence which converges to a solution
$g\in\Mloc^\infty\left(\R^n,[0, \infty)\right)$ where the convergence
$\upi g\to g$ is uniform on $B_j \timess [0,T] $ for every
fixed $j \in \N$ and $T \in (0, \infty)$ and $g(0) = g_0.$
\end{proof}

%%%%%%%%%%%%%%%%%%%%%%%%%%%%%%%%%%%%%%%%%%%%%%%%%%%%%%%%%%%%
\section{Interior Estimates}\label{int est sec}

Note that the following lemma does not imply
$\lambda_i(x,t)-1\to0$ for $|x|\to\infty$ uniformly in $t$, 
i.\,e. 
$$\lim\limits_{R\to\infty}\sup\limits_{x\in\R^n\setminus B_R}
\sup\limits_{i\in\{1,\,\ldots,\,n\}}\sup\limits_{t\in[0,\infty)}
|\lambda_i(x,t)-1|$$
may be nonzero. 

\begin{lemma}\label{decay lem}
  Let $g\in\Mloc^\infty\left(\R^n,(0,\,\infty)\right)$ be a solution
  to \eqref{DeTurck flow} with $g_0$ as in Theorem \ref{thm conv}.
  Then the eigenvalues $(\lambda_i(x,t))$ of $(g_{ij}(x,t))$ with
  respect to $(h_{ij}(x))$ uniformly tend to $1$ in bounded time
  intervals, 
$$\lim\limits_{|x|\to\infty}\lambda_i(x,t)=1.$$  
\end{lemma}

This Lemma is a direct consequence of the following interior closeness
estimate.

\begin{lemma}\label{interiour fair}
  Let $g\in\Mloc^\infty\left(\R^n,(0,\infty)\right)$ be a solution to
  \eqref{DeTurck flow} which is $\tilde\epsilon(n)$-close to the
  Euclidean background metric $h$. Assume that $g_0$ is
  $\epsilon$-close to $h$ on some ball $B_R(x_0)$ for some $\epsilon
  \leq\tilde\epsilon(n)/2$. Then there is a constant $\gamma(n)>0$
  such that $g(t)$ is $2\epsilon$-close to $h$ on $B_{R/2}(x_0)$ for
  $t\in[0,\gamma R^2]$.
\end{lemma}

\begin{proof} 
According to the estimates in Section \ref{evol eq sec}, there
exists $\epsilon(n)>0$ and $m(n) \in \N$ such that
$$\fracp{}t(\phi_m+\psi_m-2n)\leq
g^{ab}\gradh_a\gradh_b(\phi_m+\psi_m
-2n)-\sum\limits_{i,\,j,\,k}\left(\gradh_kg_{ij}\right)^2\ . $$

If the background metric $h$ is the standard flat metric on $\R^n$,
we define for $\lambda>0$ 
$$g^\lambda(x,t)=g\left(\frac x\lambda,\frac t{\lambda^2}\right).$$
If $g$ is a solution to the $h$-flow, $g^\lambda$ is also a
solution to the $h$-flow as $h$ is the standard metric on $\R^n$. 
Moreover, $g^\lambda$ is $\epsilon$-close to $h$, if $g$ is
$\epsilon$-close to $h$. 

By scaling  
with the factor $\lambda=1/R$ and translation we can assume
that $R=1$ and $x_0=0$.  Let $\eta \in C^\infty_c(B_1)$ be such that
$0\leq \eta\leq 1$ and $\eta \equiv 1$ on $B_{1/2}$ and let
$$\zeta:=\eta \cdot (\phi_m+\psi_m-2n)\ .$$
We compute
\begin{align*} \fracp{}t\zeta\leq&\,\ 
g^{ab}\gradh_a\gradh_b \zeta - 2 g^{ab}\gradh_a \eta
\gradh_b(\phi_m+\psi_m-2n) \\
&\,- (\phi_m+\psi_m-2n) g^{ab}\gradh_a\gradh_b \eta
- \eta \sum\limits_{i,\,j,\,k}\left(\gradh_kg_{ij}\right)^2\\
\leq&\,\  g^{ab}\gradh_a\gradh_b \zeta + c\ , 
\end{align*} 
since $\big|\gradh(\phi_m+\psi_m-2n)\big|^2 \leq c
 \sum\limits_{i,\,j,\,k}\left(\gradh_kg_{ij}\right)^2$ and $|\nabla
 \eta|^2/\eta \leq C(n,\Vert\eta\Vert_{C^2})$. So by the
 maximum principle the maximum of $\zeta$ grows at most linearly,
 which implies the stated estimate. \end{proof}

By scaling we want to extend a priori derivative estimates for
the metric to balls of any radius. If the background metric $h$ 
is not necessarily flat, the third author \cite{MilesC0}
obtained the following
\begin{theorem}\label{Miles int est}
  Let $R>0$. Let $h$ be a complete background metric of bounded
  curvature. Fix $T=T(n,R,h)>0$ sufficiently small.  Fix a point
  $x_0$.  Let $g$ be a solution to the $h$-flow on
  ${}^{{}^h}\!B_R(x_0) \timess(0,T)$ which is $\tilde\epsilon$-close
  to the background metric $h$ for
  $\tilde\epsilon=\tilde\epsilon(n)>0$ fixed sufficiently
  small, where ${}^{{}^h}\!B_R(x_0)$ denotes a geodesic ball of radius
  $R$ with respect to the metric $h$. Then
$$\left|\gradh^ig(x,t)\right|\leq\frac{c(n,i,R)}{t^{i/2}}$$
for all $(x,t)\in{}^{{}^h}\!B_{R/2}(x_0)\timess(0,T)$ and all $i\in\N$. 
\end{theorem}

 Applying Theorem \ref{Miles int est} to $g^\lambda$ as defined
in the proof of Lemma \ref{interiour fair}, we get directly
\begin{corollary}\label{int est}
Let $h$ denote the standard flat metric on $\R^n$ and assume that
$\epsilon=\epsilon(n)>0$ and $\gamma=\gamma(n)>0$ 
are chosen sufficiently small. If 
$g:\R^n\timess(0,T)$ solves the $h$-flow in $B_R(x_0)\timess
\left(0,\gamma R^2\right)$ for some $x_0\in\R^n$, $R>0$, 
and if $g$ is $\epsilon$-close to $h$ on
$B_R(x_0)\timess\left(0,\gamma R^2\right)$, then we
get the a priori estimates 
$$\left|\gradh^ig(x,t)\right|\leq\frac{c(n,i)}
{t^{i/2}}$$
for all $(x,t)\in B_{R/2}(x_0)\timess\left(0,\gamma R^2\right)$ 
and all $i\in\N$. We set $c_1:=c(n,1)$.
\end{corollary}

\begin{proof}[Proof of Theorem \ref{thm one}]
   Let us first assume that
  $g_0\in\Mloc^\infty\left(\R^n\right)$.  Theorem \ref{existence thm}
  implies that an $\epsilon$-close solution
  $g\in\Mloc^\infty\left(\R^n,(0,\infty)\right)$ to \eqref{DeTurck
    flow} exists.  The interior decay estimates of Corollary \ref{int
    est} and Arzelà-Ascoli imply that the metrics $g(t)$ converge
  subsequentially in $\Mloc^\infty$ to a complete flat metric
  as $t\to\infty$.

  We approximate $g_0\in\M_0\left(\R^n\right)$ by smooth metrics $\upi
  g_0$ preserving the $\epsilon_0$-closeness. Thus we obtain
  solutions $\upi g\in\Mloc^\infty$ which are $\epsilon$-close to $h$.
  In view of the interior a priori estimates we obtain a limiting
  solution $g\in\Mloc^\infty\left(\R^n,(0,\infty)\right)$.

  Note that $g(t)\to g_0$ in $\Mloc^0$ as $t\searrow0$: Fix a point
  $x_0\in\R^n$ and use $g_0(x_0)$ as a flat background metric on
  $\R^n$. Thus the metrics $\upi g$ are again solutions to the
  $h$-flow with $h=g_0(x_0)$. Hence we can apply the interior
  closeness estimates of Lemma \ref{interiour fair} in order to see
  that all metrics $\upi g$ attain their initial values uniformly in
  $i$. That implies that $g(t)\to g_0$ as $t\searrow0$ locally
  uniformly.
\end{proof}

%%%%%%%%%%%%%%%%%%%%%%%%%%%%%%%%%%%%%%%%%%%%%%%%%%%%%%%%%%%%
\section{Integral Estimates}\label{integral est sec}
In this section, we are once again concerned with the 
quantity $\phi_m+\psi_m-2n$ as introduced in \eqref{phi psi 2n sum}.

\begin{theorem}\label{int delta thm}
  Fix $\delta>0$, $m=m(n)\in\N$, and $p\geq 1$.  Let
  $g\in\Mloc^\infty\left(\R^n,(0,\infty)\right)$ be a solution to
  \eqref{DeTurck flow} which is $\tilde\epsilon(n,m)$-close to the
  standard Euclidean metric $h$ for some
  $1\geq\tilde\epsilon=\tilde\epsilon(n,m)>0$ sufficiently small.
  Let $g_0$ be as in Theorem \ref{thm conv}.  Then the integral
$$I(t)\equiv I^{m,p}_{\delta}(t) 
:=\frac1p\int\limits_{\R^n}(\phi_m+\psi_m-2n-\delta)_+^p\,dx 
\equiv\frac1p\int\limits_{\R^n}\Phi^p_{m,\delta}$$
is non-increasing in time.
\end{theorem}
\begin{proof}
  As $g$ is $\tilde\epsilon$-close to $h$, $\phi_m+\psi_m-2n$ is
  uniformly bounded above.  According to Lemma \ref{decay lem}, we see
  that $\phi_m+\psi_m-2n\to0$ for $|x|\to\infty$ and $t$ bounded
  above.  Thus the positive part of $\phi_m+\psi_m-2n-\delta$ is
  contained in a compact set for every bounded time interval and
  $I(t)$ is finite there.

Assume that in appropriate coordinates, we have at a fixed point
$h_{ij}=\delta_{ij}$, 
$g_{ij}=\diag(\lambda_1,\,\lambda_2,\,\ldots,\,\lambda_n)$,
$\lambda_i>0$.

Recall from Section \ref{evol eq sec} that we can estimate
in the evolution equation of $\phi_m+\psi_m$ for metrics
$\epsilon$-close to $h$ as follows
\begin{align*}
\fracp{}t(\phi_m+\psi_m)\leq&\,g^{ij}\gradh_i\gradh_j
(\phi_m+\psi_m)\\
&\,-m(m-1)(1+\epsilon)^{-(m-1)}
\sum\limits_{i,\,j,\,k}\left(\gradh_kg_{ij}\right)^2\\
&\,+c(n)m(1+\epsilon)^{m+1}
\sum\limits_{i,\,j,\,k}\left(\gradh_kg_{ij}\right)^2.
\end{align*}

We now want to show that $I(t)$ is decreasing in time for $m$ and 
$\varepsilon$ chosen properly. Note that $\Phi^p_{m,\delta}$ is 
Lipschitz continuous in space and time, and the support of
$\Phi^p_{m,\delta}$ is contained in $B_R(0)$ for some $R>0$ on bounded
time intervals. This yields that
$I(t)$ is Lipschitz continuous in
time as well and thus also absolutely continuous. By Sard's theorem we 
know that for almost every
$\delta>0$ the sets 
$$U_\delta\equiv 
U_\delta(t):=\{x\in\R^n\ |\ \phi_m(x,t)+\psi_m(x,t)-2n>\delta\}$$ 
have a smooth boundary for almost every $t$. For such a $\delta$ and 
$0\leq t_1< t_2$, we can compute
\begin{align*}
 I(t_2)- I(t_1)= &\, \int\limits_{t_1}^{t_2}\frac1p
\frac{d}{d\tau} \int\limits_{\R^n}\Phi_{m,\delta}^p\,
dx\, d\tau\\
= &\,  \int\limits_{t_1}^{t_2} 
\int\limits_{U_\delta}\Phi_{m,\delta}^{p-1}\fracp{}{\tau}
(\phi_m+\psi_m)\,dx\, 
d\tau\ ,
\end{align*}
as $\Phi_{m,\delta}$ is
compactly supported as long as $t$ is finite. Since the boundary of
$U_\delta$ is smooth we may estimate and integrate by parts
\begin{align*}
 I(t_2)-I(t_1)\leq &\, 
\int\limits_{t_1}^{t_2}\int\limits_{U_\delta}\Phi_{m,\delta}^{p-1}
\left(g^{ij}\gradh_i\gradh_j
(\phi_m+\psi_m)\right)\ dx\, d\tau\umbruch \\
&\,-\int\limits_{t_1}^{t_2}\int\limits_{U_\delta}
m(m-1)(1+\epsilon)^{-(m-1)}\Phi_{m,\delta}^{p-1}\gradsquared\ dx\, 
d\tau\umbruch \\
&\,+\int\limits_{t_1}^{t_2}\int\limits_{U_\delta}
c(n)m(1+\epsilon)^{m+1}\Phi_{m,\delta}^{p-1}\gradsquared\ dx\, 
d\tau\umbruch \\
= &\, \int\limits_{t_1}^{t_2}
\int\limits_{\partial U_\delta} \Phi_{m,\delta}^{p-1} \nu_ig^{ij}\gradh_j
\Phi_{m,\delta}\, d\sigma\,dt \umbruch\\
&\, -\int\limits_{t_1}^{t_2}\int\limits_{U_\delta}(p-1)
\Phi_{m,\delta}^{p-2}g^{ij}\gradh_i\Phi_{m,\delta}\gradh_j
\Phi_{m,\delta}\, dx\, d\tau \umbruch\\
&\, -\int\limits_{t_1}^{t_2}\int\limits_{U_\delta}
\Phi_{m,\delta}^{p-1}\Big(\gradh_i g^{ij}\Big)
\gradh_j\Phi_{m,\delta} \, dx\, d\tau \umbruch \\
&\, + \int\limits_{t_1}^{t_2}\int\limits_{U_\delta}
(-m(m-1))(1+\epsilon)^{-m+1}
\Phi_{m,\delta}^{p-1}\gradsquared\ dx\, d\tau\umbruch \\
&\,+\int\limits_{t_1}^{t_2}\int\limits_{U_\delta}
c(n)m(1+\epsilon)^{m+1}\Phi_{m,\delta}^{p-1}
\gradsquared\ dx\, d\tau\umbruch \\
\leq &\,\ m\int\limits_{t_1}^{t_2}\int\limits_{U_\delta}\Phi_{m,\delta}^{p-1}
\gradsquared\cdot\\
&\,\qquad\qquad\cdot\left(-(m-1)(1+\epsilon)^{-m+1}
+c(n)(1+\epsilon)^{m+3}\right)\, dx\, d\tau\ ,
\end{align*}
where we used that ${}^{\up h}\!\left|\gradh\Phi_{m,\delta}\right|
\leq c\,m\,(1+\epsilon)^{m+1}\cdot {}^{\up h}\!\left|\gradh
  g_{ij}\right|$ and that $\gradh \Phi_{m,\delta}$ is
antiparallel to the outer unit normal $\nu$ of $U_\delta$ along
$\partial U_\delta$. Note that the integration by parts involving the
exponent $p-2$ above is justified by applying the divergence
theorem on sets $U_{\delta_i}$ for a sequence $\delta_i \searrow
\delta$, where we can assume by Sard's theorem that $\partial
U_{\delta_i}$ is smooth for all $i$ and that $\partial U_{\delta_i}\to
\partial U_{\delta}$. 
If $m$ and $\epsilon$ are such that 
$$1+c(n)(1+\epsilon)^{2m+2}\leq m,$$
we obtain that the right hand side is nonpositive and our theorem
follows for such a $\delta$. Fix $\tilde\epsilon>0$
accordingly. Since
$$I_{m,p}^{\delta_i}(t)\to I_{m,p}^{\delta}(t)$$ 
for $\delta_i \to \delta$ we obtain the above monotonicity for all
$\delta$.
\end{proof}
\begin{corollary}
  Let $g\in\Mloc^\infty\left(\R^n,(0,\infty)\right)$ be a solution to
  \eqref{DeTurck flow} with $g_0$ as in Theorem \ref{thm conv}. Let
  $m$, $p$, and $\tilde\epsilon$ be as in Theorem \ref{int
    delta thm}.  If $I_0^{m,p}(0)$ is finite, then $I_0^{m,p}(t)$ is
  non-increasing and thus
$$I^{m,p}(t):=\frac1p\int\limits_{\R^n}
(\phi_m+\psi_m-2n)^p\,dx=
I_0^{m,p}(t)\leq I_0^{m,p}(0)<\infty.$$
\end{corollary}
\begin{proof}
Theorem \ref{int delta thm} implies that
$$I^{m,p}_\delta(t)\leq I^{m,p}_\delta(0)
\leq I^{m,p}_0(0)$$
for every $\delta>0$. For $\delta\searrow0$, we get
$$I^{m,p}_\delta(t)\to I^{m,p}_0(t)$$
and our claim follows. 
\end{proof}

%%%%%%%%%%%%%%%%%%%%%%%%%%%%%%%%%%%%%%%%%%%%%%%%%%
\section{Convergence of Eigenvalues}\label{eigenval conv sec}
\begin{lemma}\label{lambda decay no rate}
  Let $g\in\Mloc^\infty\left(\R^n,(0,\infty)\right)$ be a solution to
  \eqref{DeTurck flow} with $g_0$ as in Theorem \ref{thm conv}, such
  that $g(t)$ is $\tilde\epsilon(n)$-close to $h$ for all $t$.  Then
  the eigenvalues $(\lambda_i)_{1\leq i\leq n}$ of $g(t)$ with respect
  to the background metric $h$ converge uniformly to one as
  $t\to\infty$,
$$\sup\limits_{x\in\R^n}|\lambda_i(x,t)-1|\to0
\quad\text{as }t\to\infty.$$
\end{lemma}
\begin{proof}
Denote the eigenvalues of $g_{ij}$ with respect to
$h_{ij}$ by $(\lambda_i)$. It is convenient
to assume that $\lambda_1\leq\ldots\leq\lambda_n$. 

Assume that the lemma were false. Then we could find $\zeta>0$ and a
sequence $(x_k,t_k)_{k\in\N}$ in $\R^n\timess[0,\infty)$ with
$t_k\to\infty$ so that
\begin{equation}\label{eigenval deviate}
\max\limits_{i\in\{1,\,\ldots,\,n\}}|\lambda_i(x,t)-1|\geq2\zeta.
\end{equation}

The strategy of the proof is to use Corollary \ref{int est} with $i=1$
to show that we find a controlled spatial neighborhood of $(x_k,t_k)$
such that \eqref{eigenval deviate} is fulfilled with $2\zeta$ replaced
by $\zeta$, i.\,e. the eigenvalues are not all close to one
there.  Therefore, $\phi_m+\psi_m-2n$ is estimated from below by a
positive constant in that neighborhood. For large values of $t$, these
neighborhoods can be chosen arbitrarily large.  Thus $I^{m,p}_\delta$
becomes large, if $\delta$ is chosen small enough. This contradicts
Theorem \ref{int delta thm}.

Here we present the details: let $B_R(x_k)$ denote a 
Euclidean ball around $x_k$. According to Corollary 
\ref{int est}, applied with $i=1$, we obtain
for $x\in\tilde B_R(x_k)$ that the eigenvalues 
$(\lambda_i(x,t_k))$ differ at most by 
$c\cdot R\cdot t_k^{-1/2}$ from the eigenvalues
$(\lambda_i(x_k,t_k))$, 
$$\sup\limits_{i\in\{1,\,\ldots,\,n\}}|\lambda_i(x_k,t_k)-
\lambda_i(x,t_k)|\leq\frac{c_1 R}{\sqrt{t_k}}.$$

We deduce that in a ball of radius at least $R$ 
around $x_k$ with 
$R:=\frac1{c_1}\zeta\sqrt{t_k}$, we have 
$$\max\limits_{i\in\{1,\,\ldots,\,n\}}|\lambda_i(x,t_k)-1|\geq\zeta.$$
Note that this implies
$$\max\limits_{i\in\{1,\,\ldots,\,n\}}
\left|\lambda_i^m(x,t)-1\right|\geq\zeta.$$

So at least one eigenvalue differs significantly from one in that
ball. As our solution is $\tilde\epsilon(n)$-close to $h$ with
$0<\tilde\epsilon(n)\le1$, we still have
$$\tfrac12\leq(1+\tilde\epsilon)^{-1}\leq\lambda_i,\,\lambda_i^{-1}
\leq1+\tilde\epsilon\leq2.$$

We obtain in $B_R(x_k)\timess\{t_k\}$
$$\phi_m+\psi_m-2n=\sum\limits_{i=1}^m
\frac1{\lambda_i^m}\left(\lambda_i^m-1\right)^2
\geq\frac1{2^m}\zeta^2.$$

We define $\delta:=\frac12\frac1{2^m}\zeta^2$ and get in
$B_R(x_k)\timess\{t_k\}$
\begin{align*}
\phi_m+\psi_m-2n\geq2\delta,\umbruch\\
(\phi_m+\psi_m-2n-\delta)_+\geq\delta.
\end{align*}
This allows to estimate $I^{m,p}_\delta(t_k)$ from below
as follows
\begin{align*}
I^{m,p}_\delta(t_k)\geq\, &\,\quad\tfrac1p\!\!\!\!\!\!\!\!
\int\limits_{B_{\frac1{c_1}\zeta\sqrt{t_k}}(x_k)}\!\!\!\!\!\!\!\!\!  
(\phi_m+\psi_m-2n-\delta)_+^p\umbruch\\
\geq&\,\,\frac1p\cdot\delta^p\cdot c(n)\cdot
\left(\frac1{c_1}\zeta\sqrt{t_k}\right)^n.
\end{align*}
We deduce that $I^{m,p}_\delta(t_k)\to\infty$ for
$t_k\to\infty$. This contradicts Theorem \ref{int delta thm}
which implies that 
$$I^{m,p}_\delta(t_k)\leq I^{m,p}_\delta(0)$$
as our initial decay assumption guarantees that for fixed
$m\in\N_+$ as in Theorem \ref{int delta thm}, 
$p\geq1$, and $\delta>0$, the integral $I^{m,p}_\delta(0)$
is finite. 
\end{proof}

\begin{proof}[Proof of Theorem \ref{thm conv}]
Lemma \ref{lambda decay no rate} implies that $g(t)\to h$
in $C^0$ as $t\to\infty$. The convergence of the derivatives
$\gradh^m g(t)$ for all $m\in\N$ 
is a consequence of Corollary \ref{int est}.
\end{proof}

If $I^{m,p}_0(0)$ is finite for some $m$, $p$ as in Section
\ref{integral est sec}, such a calculation implies even a decay rate
in time of
$$\sup\limits_{x\in\R^n}|\lambda_i(x,t)-1|.$$

\begin{lemma}\label{lambda decay with rate}
  Let $g\in\Mloc^\infty\left(\R^n,(0,\infty)\right)$ be a solution to
  \eqref{DeTurck flow} with $g_0$ as in Theorem \ref{thm conv}, such
  that $g(t)$ is $\tilde\epsilon(n)$-close to $h$ for all $t$.  Then
  the eigenvalues $(\lambda_i)_{1\leq i\leq n}$ of $g(t)$ with respect
  to the flat background metric $h$ converge to $1$ as $t\to\infty$:
$$\sup\limits_{x\in\R^n}\sup\limits_{i\in\{1,\,\ldots,\,n\}}
|\lambda_i(x,t)-1|\leq c(n,m,p,c_1)
\cdot\left(I^{m,p}_0(0)\right)^{\frac1{2p+n}}
\cdot\left(\frac1t\right)^{\frac n{2(2p+n)}}$$
if $I^{m,p}_0(0)<\infty$.
\end{lemma}
\begin{proof}
Fix $x_0\in\R^n$, $t_0>0$, and $i_0\in\{1,\,\ldots,\,n\}$. Define
$$\zeta:=\tfrac12|\lambda_{i_0}(x_0,t_0)-1|.$$
We want to estimate $\zeta$ from above. As in the proof of Lemma
\ref{lambda decay no rate}, we obtain that
\begin{align*}
\zeta\leq&\,\max\limits_{i\in\{1,\,\ldots,\,n\}}
\max\limits_{x\in B_{\frac1{c_1}\zeta\sqrt{t_0}}(x_0)}
|\lambda_i(x,t_0)-1|,\umbruch\\
\phi_m+\psi_m-2n\geq&\,\frac1{2^m}\zeta^2,\umbruch\\
I^{m,p}_0(0)\geq&\,I^{m,p}_0(t_0)\umbruch\\
\geq&\,\frac1p\int\limits_{B_{\frac1{c_1}\zeta\sqrt{t_0}}(x_0)}  
(\phi_m+\psi_m-2n)^p\\
\geq&\,\frac1p\left(\frac1{2^m}\zeta^2\right)^p\cdot c(n)\cdot
\left(\frac1{c_1}\zeta\sqrt{t_0}\right)^n,\umbruch
\intertext{and thus}
(2\zeta)^{2p+n}\leq&\,c(n,m,p,c_1)\cdot
t_0^{-n/2}\cdot I_0^{m,p}(0).
\end{align*}
Our claim follows.
\end{proof}

\section{Improved $C^1$-Estimates}\label{imp est sec}
Based on an iteration of two steps, we can improve our a priori
estimates. So far, we have obtained a priori estimates of the form
\begin{align}\nonumber
|g-h|_{C^0}\leq&\,\frac{c\left(n,m,p,c_1(n),I^{m,p}_0\right)}
{t^{\frac n{2(2p+n)}}}\equiv\frac{c_0}{t^{\frac n{2(2p+n)}}},
\intertext{see Lemma \ref{lambda decay with rate}, and}
\Big|\gradh^i g\Big|\leq&\,\frac {c(i,n)}{t^{i/2}}
\equiv\frac{c_i}{t^{i/2}}
\label{orig Ci est}
\end{align}
for every $i\in\N$, see Lemma \ref{int est}. 
The two steps to improve these are:

\textit{Step 1:}
Interpolation inequalities of the form
$\Vert Du\Vert_{L^\infty}^2\le c\cdot\Vert u\Vert_{L^\infty}
\cdot\Vert D^2u\Vert_{L^\infty}$ and an 
iteration argument as in \cite[Lemma C.2]{OSKnutAIHP} 
can be applied to a metric $g$ satisfying \eqref{orig Ci est} and
$$|g-h|_{C^0}\leq\frac a{t^\alpha}$$
for some constant $a$.
We obtain for any $0<\beta<\alpha$ the estimate
$$\Big|\gradh^i g\Big|\leq\frac {c((c_k),a,\alpha,\beta)}
{t^{\frac i2+\beta}}.$$

\textit{Step 2:}
Recall that for solutions of the $h$-flow satisfying
$$\left|\gradh g\right|\leq\frac c{t^\gamma}.$$
Arguing as in Lemma \ref{lambda decay with rate} yields
$$|g-h|_{C^0}\leq\frac c{t^{\gamma\frac n{2p+n}}}.$$

\textit{Iteration:} We start with 
$$\left|\gradh g\right|\leq\frac c{t^{\gamma_0}}.$$
According to Corollary \ref{int est} (with $i=1$), we may
take $\gamma_0=\frac12$. So step 2 yields
$$|g-h|_{C^0}\leq\frac {C_1}{t^{\alpha_1}}$$
with $\alpha_1=\frac n{2(2p+n)}$. By combining steps 1 and 2, 
we get
$$|g-h|_{C^0}\leq\frac {C_2(\delta)}{t^{\alpha_2}}$$
with $\alpha_2=\left(\frac12+(1-\delta)\alpha_1\right)
\frac n{2p+n}$ for any fixed $0<\delta<1$. We also get
$$|g-h|_{C^0}\leq\frac {C_k(\delta)}{t^{\alpha_k}}$$
with $\alpha_{k+1}=\left(\frac12+(1-\delta)\alpha_k\right)
\frac n{2p+n}$ for all $k\in\N$.  

It is easy to check that 
$$\frac n{2(2p+n)}=\alpha_1<\alpha_2<...<\frac n{2(2p+\delta n)}.$$
As $\alpha_k$ is increasing in $k$, we see that $\lim\alpha_k$
exists. Passing to the limit in the defining equation for $\alpha_k$
yields 
$$\lim\limits_{k\to\infty}\alpha_k=\frac n{2(2p+\delta n)}.$$
Now $\delta>0$ can be chosen arbitrarily small. So we obtain 
for any $0<\zeta\ll1$
$$|g-h|_{C^0}\leq\frac{c(\zeta)}{t^{\frac n{4p}-\zeta}}.$$
We can finally apply step 1 once again in order to obtain a 
derivative bound.

Thus we have proved the following
\begin{lemma}\label{better decay estimate}
Let $g$ be as in Lemma \ref{lambda decay with rate}. Then we obtain
for every $0<\zeta\ll1$ the estimates
\begin{align}
|g-h|_{C^0}\leq&\,\frac{c(\zeta)}{t^{\frac n{4p}-\zeta}}\nonumber
\intertext{and for every $i\in\N$}
\label{goodestimate} 
\Big|\gradh^i g\Big|\leq&\,\frac{c(\zeta)}{t^{\frac i2+\frac n{4p}-\zeta}}.
\end{align}
\end{lemma}

%%%%%%%%%%%%%%%%%%%%%%%%%%%%%%%%%%%%%%%%%%%%%%%%%%%%%%%%%%%%
\section{Construction of a Solution to Ricci Flow} 
\label{const sol sec}
In this section we construct a solution to Ricci flow for smooth
initial metrics $g_0$ on $\R^n$ which are $\ep$-close to the
standard metric $\de$. It is well known that on a compact manifold, we
can recover the Ricci flow from a solution to the $h$-flow using time
dependent diffeomorphisms \cite{MilesC0,HamiltonSing}. 
We use the a priori estimates of the previous
section and construct diffeomorphisms in our set up which allow us to
construct a solution to the Ricci flow (similar to the compact case). 

\begin{lemma}\label{diffeo lem}
  Let $g\in\Mloc^\infty\left(\R^n,[0,\infty)\right)$ be a solution to
  \eqref{DeTurck flow} which is $\tilde\ep(n)$-close to $h$ (for all
  $t$ with $\tilde\epsilon$ as in Theorem \ref{thm one}). Then there
  exists a smooth family of diffeomorphisms
  $(\phi(\cdot,t))_{t\in[0,\infty)}$,
  $\phi_t\equiv\phi(\cdot,t):\R^n\to\R^n$, with $\phi(x,0)=x$, such
  that $(\phi_t)^*(g(t))$ solves \eqref{ricci flow}.
\end{lemma}

\begin{proof}
 Let ${\upi
\phi}:B_i \timess [0,\infty) \to B_i \subset \R^n$ be
the solutions to the ordinary differential equations 
$$\begin{cases}
\delt {(\upi \phi)}^{\al}(x,t) = \eta_i(\phi(x,t)){V}^{\al}(\phi(x,t),t),&
(x,t)\in B_i×[0,\infty),\\
{\upi \phi}(x,0) = x,&x\in B_i,
\end{cases}$$ where $ {V}^{\al}(y,t) := {g}^{\be \ga} \left({\up{g}
    \Gamma}^{\al}_{\be \ga} - {\uph \Gamma}^{\al}_{\be \ga}\right)
(y,t),$ and $\eta_i: \R^n \to \R$ is a smooth cut-off function $0
\leq\eta\leq 1$ with $\eta|_{B_{\frac i2}} \equiv 1$ and $\eta|_{\R^n
  - B_{i }} \equiv 0.$ Notice that ${}^iV:= \eta_i V$ is smooth and
compactly supported in $B_i$.  Hence the $\phi_i(\cdot,t):B_i \to B_i
$ are diffeomorphisms.

Using again Corollary \ref{int est}
we see that for $x\in B_{i/2}$ and small times
$$\left|\gradh g(x,t)\right|\le\frac{c(n)}{\sqrt{t}.}$$
 This implies that
$$\left|\delt {\left(\upi \phi\right)}^{\al}(x,t)\right| 
= \left|{ \upi V}^{\al}(\upi \phi(x,t),t)\right| \leq {\frac{c(n)}
  {\sqrt{t}}}$$ (here we have used that $g(t)$ is $\ep$-close to
$h$).  In particular, this gives us
\[\left|{\left(\upi \phi\right)}^{\al}(x,t)-x\right|
\leq\int\limits_0^t\left|\frac{\partial}{\partial\tau}
{\left(\upi \phi\right)}^{\al}(x,\tau)\right|\,d\tau\leq c(n)\sqrt t.\]
This implies that ${\upi \phi}(\cdot,t)|_{B_j}$ is
independent of $i$ for $i$ big enough, $j$ fixed and $t$ bounded
above. In particular we obtain, taking a (diagonal
sub-sequence) limit, a well defined solution $\phi:\R^n × [0, \infty)
\to \R^n$ to the equation 
\begin{equation}\label{evoluphi}
\begin{cases}
\delt {(\phi)}^{\al}(x,t) = {V}^{\al}(\phi(x,t),t),&
(x,t)\in \R^n\timess[0,\infty),\\
{\phi}(x,0) = x,&x\in\R^n,
\end{cases}\end{equation}
and thus  $\phi_t = \phi(t, \cdot):\R^n \to \R^n$ is a
diffeomorphism for all $t$. According to \cite{ShiJDG1989}
$(\phi_t)^*(g(t))$ solves \eqref{ricci flow}.
\end{proof}

In the situation that we have the improved decay estimates of Lemma
\ref{better decay estimate} we show that the diffeomorphisms $\phi_t$
converge to a diffeomorphism $\phi_\infty:\R^n \to \R^n$ as $t \to
\infty$ for $n \geq 3$ as long as $p\geq 1$ is small enough.

\begin{theorem}\label{diffeo thm}
  Let $g\in\Mloc^\infty\left(\R^n,[0,\infty)\right)$ be a solution to
  \eqref{DeTurck flow} and $m$, $p$, and $\tilde\epsilon(m,n)$ all be
  as in Theorem \ref{int delta thm}. Assume that \eqref{int p cond}
  holds. For $n\geq3$, $\phi_t$ converges in
  $C^\infty_{\text{\it{loc}}}$ to a diffeomorphism $\phi_\infty:\R^n
  \to \R^n$ as $t \to \infty.$
\end{theorem}
\begin{proof}
\textit{Step 1:}
Fix $p$ with $\frac n2>p\geq1$ and $\zeta>0$ such that
$\alpha:=\frac12+\frac n{4p}-\zeta>1$. 
In the following, our constants $c$ may depend on
$n$, $p$, $\zeta$, $m$, and $I^{m,p}_0$ as defined in
Theorem \ref{int delta thm}.
\par 
We wish to show that
$$|\phi_t(x) - x | \leq c \quad 
\forall \ t \in [0, \infty).$$
Corollary \ref{int est} implies that 
$$\left|\partt \phi_t(x)\right|\leq \frac c  {t^{1/2}}.$$
Hence $|\phi_t(x)-x|\leq c$ for all $0\leq t\leq1$. 
So we may assume that $t>1$.
The evolution equation \eqref{evoluphi} for $\phi_t$ and the 
fact that $|\grad g| \leq \frac c  {t^{\al}}$ imply that
$$\left|\partt \phi_t(x)\right|\leq \frac c  {t^{\al}}$$
and so we see that (upon integrating from $1$ to $t>1$)
$|\phi_1(x)-\phi_t(x)|\leq c$.
By the triangle inequality we get
$$| \phi_t(x) - x | \leq | \phi_t(x) - \phi_1(x) | + |\phi_1(x) -x |
\leq c.$$ For later use, note that this implies for $F_t:\R^n\to\R^n$
defined by $F_t:=\phi_t^{-1}$
$$| F_t(x) - x | \leq c.$$

\textit{Step 2:} Using the estimate $\left|\partt \phi_t(x)\right|\leq
\frac c {t^{\al}}$ from Step 1, we see that for fixed
$x\in\R^n$, $\phi_t(x) \to y$ as $t \to \infty$ for some $y
\in \R^n.$ We define $\phi_\infty(x) = \lim_{t \to \infty} \phi_t(x).$
All the derivatives of $\phi_t$ (for $t >1$) satisfy
$$\delt\log\left(1+\left|\uph\grad^m \phi_t(x)\right|^2\right) 
\leq\frac{c\left(\phi_1|_{B_r(y)},n,m,p,\zeta,
    (c_k),I^{m,p}_0\right)}{t^{\al}}\quad \forall \ x \in B_r(y) $$ for
some $\al > 1$ and
$$\sup_{x \in B_r(y)} \left|\uph\grad^m \phi_t(x)\right|^2 
\leq c\left(\phi_1|_{B_r(y)},n,m,p,\zeta,(c_k),I^{m,p}_0\right).$$
This may be seen as follows (for the rest of this argument, we set
$\grad:= \uph \grad$).  Differentiating gives for $V$ as in
\eqref{DeTurck flow}
\begin{align}
\left|\nabla^mV(x,t)\right|\leq&
\,c_m\sum\limits_{k=1}^{m+1}\sum\limits_{
\genfrac{}{}{0pt}{}{i_1+\cdots+i_k=m+1}
{1\leq i_1,\,\ldots,\,i_k\le m+1}}
\left|\nabla^{i_1}g\right|\cdots\left|\nabla^{i_k}g\right|,\label{oney}\\
\left|\nabla^m(V(\phi(x,t),t))\right|\leq&
\,c_m\sum\limits_{i=1}^m
\sum\limits_{\genfrac{}{}{0pt}{}{j_1+\cdots+j_i=m}
{1\leq j_1,\,\ldots,\,j_i\le m}}
\left|\left(\nabla^iV\right)(\phi(x,t),t)\right|
\cdot\left|\nabla^{j_1}\phi\right|
\cdots\left|\nabla^{j_i}\phi\right|.\label{twoey}  
\end{align}
Then
\begin{align*}
  \delt\log\left(1+\left|\grad^m \phi_t(x)\right|^2\right)=&\,
  \frac{\delt\left|\grad^m \phi_t(x)\right|^2}
  {\left(1+\left|\grad^m \phi_t(x)\right|^2\right)} \\
  \leq&\,c\,\frac{\left|\grad^m \phi_t(x)\right|\left| \grad^m \delt
      \phi_t(x)\right|}
  { \left(1+\left|\grad^m \phi_t(x)\right|^2\right)}\\
  =&\,c\, \frac{\left|\grad^m \phi_t(x)\right|\left| \grad^m
      (V(\phi_t,t))(x)\right|} { \left(1+\left|\grad^m
        \phi_t(x)\right|^2\right)}.
\end{align*}
Substituting \eqref{oney} and \eqref{twoey} into the above, using the
estimates for $\left|\nabla^k g\right|$ and arguing inductively (for
$m$) leads to the estimate.

Hence $\phi_\infty:\R^n \to \R^n$ is $C^{\infty}$ and
$\phi_t\to\phi_\infty$ converges in $C^\infty_{\text{\it{loc}}}$
as $t\to\infty$.

\textit{Step 3:}
We wish to show that  $\phi_\infty$ is a diffeomorphism.
Let $F_t$ be the inverse of $\phi_t.$
Letting $l(t):= \phi_t^{*} g(t),$ we know that
$l$ solves Ricci flow on $\R^n$.
But then we get in view of \eqref{goodestimate}
\[{}^{\displaystyle{}^{\up l }}\!\!\left|\delt l\right|(x,t) 
=2\cdot{\up l |}\Ricci(l)|(x,t)
 =2\cdot{\up g |}\Ricci(g)|(\phi_t(x),t)
\leq\frac c  { t^{\frac 3 2}} 
\quad\text{for }t\geq1,\]
which is integrable (from $1$ to $\infty$).
This implies that $l(t)$ converges
locally uniformly to a well defined continuous Riemannian metric
$l_{\infty}$ as $t \to \infty$ in view of \cite[Lemma
14.2]{HamiltonThree}.  In particular, there must exist a $c$ with $
{\frac 1 c}h\leq l(x,t) \leq c\,h$ for all $t \in
[0,\infty)$.

Now using the definition of ${l}$ we get
\begin{align*}
{\frac 1 c} \de_{\al \be} \leq l_{\al \be}(y,t)
&\,=       \parti{\phi^s}{y^{\al}}(y,t)  
\parti{\phi^k}{y^{\be}}(y,t)
\,{g}_{s k}(\phi_t(y),t) \\ 
&\,\leq  (1 +\tilde\ep)  \parti{\phi^s}{y^{\al}}
\parti{\phi^k}{y^{\be}}(y,t)
\delta_{sk}\\
&\,=  (1 +\tilde\ep)\left(D\phi\right) 
\left(D\phi\right)^T(y,t).
\end{align*}
In particular, we see that $\det\left(D\phi_t\right)^2(y)
\geq\frac1{(1 +\tilde\ep)c}>0$, where $Df$ is the Jacobian of $f$.
Taking the limit as $t \to \infty$ we get
$\det\left(D\phi_\infty\right)^2(y) \geq{\frac1{(1 +\tilde\ep)c}}>0$.
As $\det\,(D\phi_0)=1$, we see that
$\det\,(D\phi_{\infty})(y)\ge\frac1{\sqrt{(1+\tilde\epsilon)c}}>0$ for
all $y \in \R^n.$ Hence $\phi_{\infty}$ is an immersion.  In view of
the derivative estimates for $\left|\nabla^m\phi_t\right|$, the
function $\phi_\infty:\R^n\to\R^n$ is a smooth diffeomorphism if it is
bijective. Recall that $F_t\equiv F(\cdot,t)=
{\left(\phi_t\right)}^{-1}$ for $t<\infty$.  The estimates for
$\phi_t$ above ensure that there exists a function
$F_\infty:\R^n\to\R^n$ such that $F_t\to F_\infty$ in
$C^\infty_{\text{\it{loc}}}$ as $t\to\infty$.

As $\phi_t$ and $F_t$ are
smooth diffeomorphisms with estimated derivatives,
$\phi_t\to\phi_\infty$ and $F_t\to F_\infty$, we obtain
$$ (F_{\infty} \of \phi_{\infty}) (x) =  \lim_{t \to \infty} 
\left({F}_t \of { \phi}_t\right) (x) = x\quad\text{for all
}x\in\R^n.$$ Thus $\phi_\infty$ is injective. Similarly, we see that
$\phi_\infty$ is surjective. Hence $\phi_\infty$ is a diffeomorphism.
\end{proof}

\begin{proof}[Proof of Theorem \ref{thm two}]
In Lemma \ref{diffeo lem} we have constructed diffeomorphisms
$\phi_t$ so that $(\phi_t^*g(t))_{t\in[0,\infty)}$ solves
\eqref{ricci flow} if $(g(t))_{t\in[0,\infty)}$ solves
\eqref{DeTurck flow}. Define $\tilde g(t):=\phi_t^*g(t)$. 
According to Theorem \ref{thm conv}, we
get $g(t)\to h$. Theorem \ref{diffeo thm} implies that the
diffeomorphisms $\phi_t$ converge to a diffeomorphism
$\phi_\infty$ for $t\to\infty$. 
Thus $\tilde g(t)\to\phi_\infty^*h$ in $C^\infty$
as $t\to\infty$. 
\end{proof}

We can also show that the diffeomorphisms $\phi_t$ are 
close to the identity near infinity, uniformly in $t$.
\begin{lemma}
  Let $g\in\Mloc^\infty\left(\R^n,[0,\infty)\right)$ be a solution to
  \eqref{DeTurck flow} as in Theorem \ref{diffeo thm} and
  $(\phi_t)_{t\in[0,\infty]}$ be the diffeomorphisms of $\R^n$
  constructed before. Then for every $\eta> 0$ there exists an $R>0$
  such that
$$ \sup_{{\R}^n\setminus B_{R}(0)}|\phi_t(x)-x| \leq \eta $$
for all $t\in [0,\infty]$. 
\end{lemma}
\begin{proof} We first want to show that for any given $\eta^\prime>0$
  there exists an $R_1>0$ such that $g(t)$ is $\eta^\prime$-close to
  $h=\delta$ on $\R^n\setminus B_{R_1}(0)$ for all $t$.\\
  By Lemma \ref{lambda decay with rate} we can choose a $T>0$ such
  that $g(t)$ is $\eta^\prime$-close to $h$ on all of $\R^n$ for $t
  \geq T$. Now choose $R_0>0$ such that $g(0)$ is
  $\eta^\prime/2$-close to $h$ on $\R^n\setminus B_{R_0}(0)$. Applying
  the interior closeness estimates, Lemma \ref{interiour fair}, we see
  that we can find $R_1>R_0$ such that $g(t)$ is $\eta^\prime$-close
  on $(\R^n\setminus B_{R_1}(0))× [0,T]$. Note that this implies
$$\sup_{\R^n\setminus B_{R_1}(0)}|g(t)-\delta| \leq \eta^\prime $$
for all $t$. By Lemma \ref{better decay estimate} we have that
$$ \big|\uph \nabla g\big| \leq \frac{c}
{t^{\frac{1}{2}+\frac{n}{4p}-\zeta}}$$
for $t>0$. Now fix $t_1>0$ such that 
$$\int_{t_1}^\infty \frac{c}
{t^{\frac{1}{2}+\frac{n}{4p}-\zeta}}\, dt \leq \frac{\eta}{2}\ .$$
By interpolation we deduce from Lemma
\ref{int est} that 
$$ \big|\uph \nabla g\big| \leq \frac{\delta(\eta^\prime)}{\sqrt{t}} 
\quad\text{on }\R^n\setminus B_{R_1}(0)$$ for $t>0$, and in particular
for $0< t \leq t_1$, where $\delta(\eta)\to 0$ as $\eta\to 0$.
Arguing as in Step 1 of the proof of Theorem \ref{diffeo thm} we
obtain the claimed estimate for some $R>R_1$.
\end{proof}

\section{Convergence Based on Integral Bounds}
\label{conv int bounds}
\begin{proof}[Proof of Theorem \ref{thm three}]
  In view of Corollary  \ref{int est} and interpolation
  inequalities of the form $\Vert \nabla v\Vert_{C^0}\leq c\cdot \Vert
  v\Vert_{C^0}\cdot\Vert v\Vert_{C^2}$, it suffices to prove that
  $g(t)\to h$ in $C^0$ as $t\to\infty$.
\par
Assume the conclusions were false. Then we could find
$\delta>0$ and points
$$(x_k,\,t_k,\,l_k)\in\R^n×(0,\,\infty)×\{1,\,
\ldots,\,n\},$$ 
$k\in\N$, such that $t_k\to\infty$
as $k\to\infty$ and 
$$|\lambda_{l_k}(x_k,\,t_k)-1|\geq4\delta.$$
By the construction of the solution $g(t)$ in Section
\ref{existence sec} from solutions ${}^ig$, we find sequences
$(i^k_j)_{j\in\N}$ with $i^k_j\to\infty$ as $j\to\infty$, such
that for the eigenvalues $({}^i\lambda_l)_{1\leq l\leq n}$ of
${}^ig$, we have
$$\left|{}^{i^k_j}\!\lambda_{l_k}(x_k,\,t_k)-1\right|\geq3\delta.$$
Given any $R>0$, Corollary \ref{int est} implies for
$k$ sufficiently large so that $t_k$ is
sufficiently large and for $j=j(k)$ sufficiently large so that
$\nu_k-|x_k|\equiv i^k_j-|x_k|$ is sufficiently large that
\begin{equation}\label{two delta approx bound}
\inf\limits_{x\in B_R(x_k)}\sup\limits_{1\leq l\leq n}
\left|{}^{{}^{\nu_k}}\!\lambda_l(x,\,t_k)-1\right|\geq2\delta.  
\end{equation}
The condition that $\nu_k-|x_k|$ is sufficiently large
ensures that the solutions ${}^{\nu_k}g$ 
are defined on all of $B_R(x_k)$.
Define
$${}^iI^{m,p}_\delta(t):=\frac1p\int\limits_{B_i(0)}
\left({}^i\phi_m+{}^i\psi_m-2n-\delta\right)^p_+\,dx,$$
where
$${}^i\phi_m(x,t):=\sum\limits_{l=1}^n
\frac1{{}^i\lambda_l^m(x,t)}\quad\text{and}\quad
{}^i\psi_m(x,t):=\sum\limits_{l=1}^n
{}^i\lambda_l^m(x,t).$$
By construction of the initial metrics ${}^ig_0$, we have for
the eigenvalues $(\lambda_l(x))_{1\leq l\leq n}$ of the
metric $g_0$ with respect to $h$ which we label such that 
$\lambda_1(x)\leq\ldots\leq\lambda_n(x)$ and similarly 
for the eigenvalues
${}^i\lambda_1(x)\leq\ldots\leq{}^i\lambda_n(x)$ of ${}^ig_0$
for $|x|\leq i$ the inequalities
$$\begin{cases}
\lambda_l(x)\leq{}^i\lambda_l(x)\leq1&\text{if }\lambda_l(x)\leq1,\\
1\leq{}^i\lambda_l(x)\leq\lambda_l(x)&\text{if }1\leq\lambda_l(x),
\end{cases}$$
which implies in particular the estimate $\left({}^i\lambda_l^m-1\right)^2
\leq\left(\lambda_l^m-1\right)^2$ for all $i,\,l,\,m\in\N$.
We deduce that
\begin{align*}
{}^i\phi_m+{}^i\psi_m-2n=&\,\sum\limits_{l=1}^n\frac1{{}^i\lambda_l^m}
\left({}^i\lambda_l^m-1\right)^2\umbruch\\
\leq&\,\sum\limits_{l=1}^n\frac{\lambda_l^m}{\lambda_l^m}
(1+\epsilon(n))^m\left(\lambda^m_l-1\right)^2\umbruch\\
\leq&\,(1+\epsilon(n))^{2m}\sum\limits_{l=1}^n
\frac1{\lambda_l^m}\left(\lambda_l^m-1\right)^2\umbruch\\
=&\,(1+\epsilon(n))^{2m}(\phi_m+\psi_m-2n).
\end{align*}
Thus
\begin{align}\label{iImpd0 endl}
\begin{split}
{}^iI^{m,p}_\delta(0)=&\,\frac1p\int\limits_{B_i(0)}
\left({}^i\phi_m+{}^i\psi_m-2n-\delta\right)^p_+\,dx\bigg|_{t=0}\\
\leq&\,(1+\epsilon(n))^{2m}\frac1p\int\limits_{B_i(0)}
\left(\phi_m+\psi_m-2n-\delta\right)^p_+\,dx\bigg|_{t=0}\\
\leq&\,(1+\epsilon(n))^{2m}\frac1p\int\limits_{\R^n}
\left(\phi_m+\psi_m-2n-\delta\right)^p_+\,dx\bigg|_{t=0}\\
=&\,(1+\epsilon(n))^{2m}I^{m,p}_\delta(0)<\infty
\end{split}
\end{align}
is uniformly bounded in $i$.
\par
The proof of Theorem \ref{int delta thm} extends directly to our
situation and ensures that 
$${}^iI^{m,p}_\delta(t)\leq{}^iI^{m,p}_\delta(0)\quad\text{for }t\geq0.$$ 
Here, $\left({}^i\phi_m+{}^i\psi_m-2n-\delta\right)_+$ is compactly
supported in $B_i(0)$ due to the boundary values imposed on ${}^ig$ in
Theorem \ref{existence thm (Dirichlet)}. Therefore, we do not have to
use the local closeness estimates of Lemma \ref{interiour fair} in
order to justify the integration by parts employed in the proof of
Theorem \ref{int delta thm}.
\par
Using \eqref{two delta approx bound} and
\eqref{iImpd0 endl}, we estimate
for $\epsilon$ and $m$ such that $(1+\epsilon(n))^m\leq2$
as follows
\begin{align*}
  (1+\epsilon(n))^{2m}I^{m,p}_\delta(0)\geq&\,
  {}^{{}^{\nu_k}}\!I^{m,p}_\delta(0)\umbruch\\
  \geq&\,{}^{{}^{\nu_k}}\!I^{m,p}_\delta(t_k)\umbruch\\
  =&\,\frac1p\int\limits_{B_{\nu_k}\!(0)}
  \left({}^{\nu_k}\phi(x,t_k)+{}^{\nu_k}\psi(x,t_k)
    -2n-\delta\right)^p_+\,dx\umbruch\\
  \geq&\,\frac1p\int\limits_{x\in B_R(x_k)}\left(\sum\limits_{l=1}^n
    \frac1{{}^{{}^{\nu_k}}\!\lambda_l^m(x,t_k)}
    \left({}^{{}^{\nu_k}}\!\lambda^m_l(x,t_k)-1\right)^2
    -\delta\right)_+^p\,dx
  \umbruch\\
  \geq&\,\frac1p\int\limits_{B_R(x_k)}\left(\frac1{(1+\epsilon(n))^m}
    4\delta-\delta\right)^p\umbruch\\
  \geq&\,\frac1p\int\limits_{B_R(x_k)}\delta^p\umbruch\\
  =&\,\frac{\delta^p}pc(n)R^n.
\end{align*}
As $R\geq0$ can be chosen arbitrarily large, we arrive at a
contradiction. This finishes the proof of Theorem \ref{thm three}.
\end{proof}

\begin{appendix}
  \section{Conformal Ricci Flow in Two Dimensions} 
  \label{2d sec}
  Studying Ricci flow on $\R^2$ for conformally flat metrics
  $g(x,t)=e^{u(x,t)}\delta$ is equivalent
  \cite{LaniWuRtwo,HamiltonTwo} to considering the evolution equation
\begin{equation}
  \label{two d ricci}
  \delt u=\Delta_gu=e^{-u}\Delta_\delta u \quad 
  \text{on }\R^2\times(0,\infty).
\end{equation}
We are especially interested in initial conditions $u_0:\R^2\to\R$
fulfilling
\begin{equation}
  \label{u0 decay}
  \sup\limits_{\R^2\setminus B_r(0)}|u_0|\to0\quad
  \text{as }r\to\infty. 
\end{equation}
We wish to point out that we do not assume in the following that $u_0$
fulfills an $\epsilon_0$-closeness condition for a small constant
$\epsilon_0>0$, corresponding to $\sup|u_0|$ being small. It suffices
to have $\sup|u_0|$ bounded. Note also that in two dimensions, we do
not have to study Ricci harmonic map heat flow first before we can
obtain results for Ricci flow.
\par
In this situation, we get the following existence result.
\begin{theorem}\label{two d exist thm}
  Let $u_0\in C^0\left(\R^2\right)$ satisfy $\Vert
  u_0\Vert_{L^\infty}<\infty$. Then there exists $u\in
  C^\infty(\R^2\times(0,\infty))$ solving \eqref{two d ricci} such
  that $u(\cdot,t)\to u_0$ in $C^0_{\text{\it{loc}}}\left(\R^2\right)$
  as $t\searrow0$. As $t\to\infty$, $u(\cdot,t)$ converges
  subsequentially in $C^\infty_{\it{loc}}$ to a complete flat metric.
\end{theorem}
\begin{proof}[Idea of Proof]
  Mollifying $u_0$, we get a sequence $u^i_0\in
  C^\infty\left(\R^2\right)$, $i\in\N$, such that $u^i_0\to u_0$ in
  $C^0_{\text{\it{loc}}}$ as $i\to\infty$, $\left\Vert
    u^i_0\right\Vert_{L^\infty}\le\Vert u_0\Vert_{L^\infty}$, and
  $\left\Vert D^2u^i_0\right\Vert_{L^\infty}+\left\Vert
    Du^i_0\right\Vert_{L^\infty}\le c_i$. According to
  \cite{LaniWuRtwo}, there exist smooth solutions $u^i$ for $t\ge0$
  solving \eqref{two d ricci} with $u^i(\cdot,0)=u^i_0$ such that
  $\left|u^i(x,t)\right|\le\sup|u_0|$. The techniques of
  Krylov-Safonov, Schauder, and scaling imply
  $\left|D^ku^i\right|\le\frac{c_k}{t^{k/2}}$. Thus we find a smooth
  solution $u$ to \eqref{two d ricci}. Considering
  $\psi=u^i(x,t)-u^i(x_0)$, we can argue as in Lemma \ref{interiour
    fair}  (using $\psi^2$ instead of $\phi_m+\psi_m-2n$) to
  see that for any $\delta>0$ we find $i_0=i_0(\delta,u_0)$ and
  $\zeta>0$ such that $\left|u^i(x,t)-u^i(x_0)\right|\le\delta$ for
  $|x-x_0|^2+t\le \zeta$ and $i\ge i_0$. Thus $u(\cdot,t)\to u_0$ in
  $C^0_{\text{\it{loc}}}$  as $t\searrow0$. Note that in
  the case of \eqref{u0 decay}, this convergence is uniform in
  space. The remaining details are similar to the higher
  dimensional case.
\end{proof}

\begin{theorem}
  Let $u_0$ be as in Theorem \ref{two d exist thm} and assume that
  \eqref{u0 decay} is fulfilled. Let $u$ be the solution to \eqref{two
    d ricci} obtained there. Then $u(\cdot,t)\to0$ in $C^\infty$ 
as $t\to\infty$.
\end{theorem}
\begin{proof}[Idea of Proof]
  The proof is similar to the proof of Theorem \ref{thm conv}. Instead
  of Theorem \ref{int delta thm}, however, we use the following
  estimate for $p\ge\sup|u_0|+1$ (and a similar argument for the
  negative part of $u$)
  \begin{align*}
    \dt\int\limits_{\R^2}\tfrac1p(u-\delta)_+^p\,dx=
    &\,\int\limits_{\{u>\delta\}}(u-\delta)^{p-1}e^{-u}\Delta_\delta
    u\,dx\\
    =&\,\int\limits_{\{u>\delta\}}(u-\delta)^{p-2}
    ((u-\delta)-p+1)e^{-u}|\nabla u|^2\,dx\le0.
  \end{align*}
\end{proof}

\end{appendix}

\bibliographystyle{amsplain}
%\bibliography{ricci.stab} 
\def\weg#1{} \def\cprime{$'$}
\providecommand{\bysame}{\leavevmode\hbox to3em{\hrulefill}\thinspace}
\providecommand{\MR}{\relax\ifhmode\unskip\space\fi MR }
\providecommand{\MRhref}[2]{%
  \href{http://www.ams.org/mathscinet-getitem?mr=#1}{#2}
}
\providecommand{\href}[2]{#2}

\end{document}